\documentclass[graybox]{svmult}

\usepackage{type1cm}        
\usepackage{makeidx}         
\usepackage{graphicx}        
\usepackage{multicol}        
\usepackage[bottom]{footmisc}
\usepackage{newtxtext}       %
\usepackage[varvw]{newtxmath}       

\usepackage{xcolor}

\newcommand{\HFAsp}{{\boldsymbol A}}     
\newcommand{\HFAT}{{\HFAsp(HFTst)}}     
\newcommand{\HFTst}{\mathbb{T}}     
\newcommand{\HFATN}{\big( \HFAT, \, \|\HBebbes\|_\HFAsp \big)}     
\newcommand{\HBebbes}{\mbox{$\,\cdot\,$}}     
\newcommand\HFBFsp{{ \HFBsp'}}     
\newcommand{\HFBsp}{{\boldsymbol B}}     
\newcommand{\HFBspN}{(\HFBsp, \, \|\HBebbes\|_\HFBsp)}     
\newcommand{\HFCORd}{{\HFCOsp(\HFRst^d)}}     
\newcommand{\HFCOsp}{{\HFCsp_{\negthinspace 0}}}     
\newcommand{\HFRst}{{\mathbb R}}     
\newcommand{\HFCORdN}{{\big( \HFCOsp(\HFRst^d), \, \|\HBebbes\|_\infty \big)}}     
\newcommand{\HFCsp}{{\boldsymbol C}}     
\newcommand{\HFCbRd}{{\HFCbsp(\HFRdst)}}     
\newcommand{\HFCbsp}{{\HFCsp_{\negthinspace b}}}     
\newcommand{\HFRdst}{{{\HFRst^d}}}     
\newcommand{\HFCcRd}{{\HFCcsp(\HFRst^d)}}     
\newcommand{\HFCcsp}{{\HFCsp_{\negthinspace c}}}     
\newcommand{\HFCooY}{{\HFCoosp(\HFYsp)}}     
\newcommand{\HFCoosp}{{\boldsymbol{\mathcal C \negthinspace o}}}     
\newcommand{\HFYsp}{{\boldsymbol Y}}     
\newcommand\HFCspN{(\HFCsp, \, \|\HBebbes\|_\HFCsp)}     
\newcommand{\HFFAMglaLa}{{(g_\lambda)_{\HFlainLa}}}     
\newcommand{\HFlainLa}{{\lambda{\in}\Lambda}}     
\newcommand{\HFFF}{{\mathcal{F}}}     
\newcommand\HFFLi{{\mathcal F}{\negthinspace \Lisp}}     
\newcommand{\Lisp}{{\HFLsp^1}}     
\newcommand{\HFFLiRd}{{ \HFFLi(\HFRdst) }}     
\newcommand\HFFLiRdN{\big( \HFFLiRd, \, \|\HBebbes\|_{\HFFLisp} \big)}     
\newcommand{\HFFLisp}{{ {\mathcal F} \negthinspace \Lisp}}     
\newcommand\HFFLinf{{\mathcal F} \HFnnth \HFLinsp}     
\newcommand{\HFnnth}{{ \negthinspace \: \negthinspace }}     
\newcommand{\HFLinsp}{{\HFLsp^\infty}}     
\newcommand{\HFFLinsp}{{\mathcal F}\HFnth \HFLinsp}     
\newcommand\HFnth{\negthinspace}     
\newcommand\HFFLpsp{{\HFFF \negthinspace \HFLpsp}}     
\newcommand{\HFLpsp}{{\HFLsp^p}}     
\newcommand{\HFFT}{{\operatorname{{\mathcal F}}}}     
\newcommand{\HFHRdsp}{ {\mathcal H}_\HFRdst}     
\newcommand\HFHilb{\mathcal H}     
\newcommand\HFHs{{{\cal H}_s}}     
\newcommand\HFHsRd{{{\mathcal H}_s (\HFRdst)}}     
\newcommand\HFHsRdN{{({\cal H}_s (\HFRdst), \, \|\HBebbes\|_{\HFHs} \big)}}     
\newcommand{\HFIFT}{\operatorname{\mathcal F}^{-1}}     
\newcommand{\HFLiRd}{{\Lisp \HFnth (\HFRst^d)}}     
\newcommand{\HFLiRdN}{\big( \HFLiRd, \, \|\HBebbes\|_1 \big)}     
\newcommand\HFLiTFd{ \Lisp(\HFTFd)}     
\newcommand{\HFTFd}{{{ \HFRdst \times \HFRdsth }}}     
\newcommand\HFLinf{{\HFLsp^\infty}}     
\newcommand{\HFLsp}{{\boldsymbol L}}     
\def\Lino#1{{ \| #1 \|_\Lisp}}     
\newcommand{\HFLpR}{{\HFLpsp(\HFRst)}}     
\newcommand{\HFLpRd}{{\HFLpsp(\HFRst^d)}}     
\newcommand{\HFLpRdN}{\big( \HFLpRd, \, \|\HBebbes\|_p \big)}     
\newcommand{\HFLpwRd}{{\HFLpwsp(\HFRst^d)}}     
\newcommand{\HFLpwsp}{{\HFLsp^p_{\negthinspace w}}}     
\newcommand{\HFLqRd}{{ \HFLsp^q(\HFRdst) }}     
\newcommand{\HFLqRdN}{\big( \HFLqRd, \, \|\HBebbes\|_q \big)}     
\newcommand{\HFLtR}{{\HFLtsp(\HFRst)}}     
\newcommand{\HFLtsp}{{\HFLsp^2}}     
\newcommand{\HFLtRN}{\big( \HFLtR, \, \|\HBebbes\|_2 \big)}     
\newcommand{\HFLtRd}{{\HFLtsp(\HFRst^d)}}     
\newcommand{\HFLtRdN}{\big( \HFLtRd, \, \|\HBebbes\|_2 \big)}     
\newcommand{\HFLtRtd}{{\HFLtsp(\HFRst^{2d})}}     
\newcommand{\HFLtcG}{{\HFLtsp(\HFcG)}}     
\newcommand{\HFcG}{\mathscr{G}}     
\newcommand{\HFLtcGN}{\big( \HFLtcG, \, \|\HBebbes\|_2 \big)}     
\newcommand\HFLtvsRd{{\HFLsp^2_{v_s}(\HFRdst)}}     
\newcommand{\HFLtwRd}{{\HFLtwsp(\HFRst^d)}}     
\newcommand{\HFLtwsp}{{\HFLsp^2_{\negthinspace w}}}     
\newcommand\HFLtwRdN{\big( \HFLtwRd, \, \|\HBebbes\|_{2,w} \big)}     
\newcommand{\HFMbRd}{{\HFMbsp(\HFRst^d)}}     
\newcommand{\HFMbsp}{{\HFMsp_{\negthinspace b}}}     
\newcommand\HFMbRdN{{(\HFMbsp(\HFRst^d), \| \HBebbes \|_\HFMbsp )}}     
\newcommand{\HFMsp}{{\boldsymbol M}}     
\newcommand\HFMinf{{\HFMsp^\infty}}     
\newcommand{\HFMisp}{{ \HFMsp^1 }}     
\newcommand{\HFMpRd}{{\HFMpsp(\HFRst^d)}}     
\newcommand\HFMpsp{{\HFMsp^p}}     
\newcommand{\HFMppRd}{{ \HFMsp^{p,p}(\HFRdst)}}     
\newcommand{\HFMqRd}{{\HFMqsp(\HFRst^d)}}     
\newcommand{\HFMqsp}{{\HFMsp^q}}     
\newcommand{\HFNst}{{\mathbb N}}     
\newcommand{\HFPgLam}{{\HFPfr_{g,\Lambda}}}
\newcommand\HFPsifam{{ \Psi = (\psi_i)_{i \in I} }}     
\newcommand{\HFRaRd}{{\HFRasp(\HFRst^d)}}     
\newcommand{\HFRasp}{{\boldsymbol{\mathcal R}}}     
\newcommand{\HFRdsth}{{\widehat{\HFRst}^d}}     
\newcommand{\HFRtdst}{{\HFRst^{2d}}}     
\newcommand{\HFSINC}{\operatorname{SINC}}     
\newcommand{\HFSOG}{{\HFSOsp(G)}}     
\newcommand{\HFSOsp}{{\HFSsp_{\negthinspace 0}}}     
\newcommand{\HFSOGN}{\big( \HFSOG, \|\HBebbes\|_\HFSOsp \big)}     
\newcommand{\HFSOGTrRd}{{ (\HFSOsp,\HFLtsp,\HFSOPsp)(\HFRdst) }}     
\newcommand{\HFSOPsp}{{\HFSsp_{\negthinspace 0}'}}     
\newcommand{\HFSOPRd}{{\HFSOPsp(\HFRst^d)}}     
\newcommand{\HFSOPRdN}{(\HFSOPRd , \| \HBebbes \|_{\HFSOPsp} ) }     
\newcommand{\HFSOPRtd}{{\HFSOPsp(\HFRst^{2d})}}     
\newcommand{\HFSOPnorm}[1]{{\lVert #1 \rVert_\HFSOPsp}}     
\newcommand\HFSOPsi{{\HFSOPnorm{\sigma}}}     
\newcommand{\HFSsp}{{\boldsymbol S}}     
\newcommand{\HFSORd}{{\HFSOsp(\HFRst^d)}}     
\newcommand{\HFSORdN}{\big( \HFSORd, \|\HBebbes\|_\HFSOsp \big)}     
\newcommand{\HFSORtd}{{\HFSOsp(\HFRst^{2d})}}     
\newcommand{\HFSOnorm}[1]{{\lVert #1 \rVert_\HFSOsp}}     
\newcommand{\HFScPRd}{{\HFScPsp(\HFRst^d)}}     
\newcommand{\HFScPsp}{{\HFScsp'}}     
\newcommand{\HFScsp}{{\boldsymbol{\mathcal S}}}     
\newcommand{\HFScRd}{{\HFScsp(\HFRst^d)}}     
\newcommand{\HFVgf}{{V_g f}}     
\newcommand\HFVgg{{V_gg}}     
\newcommand{\HFWBC}{{\HFWsp(\HFBsp,\HFCsp)}}     
\newcommand{\HFWsp}{{\boldsymbol W}}     
\newcommand{\HFWBCN}{\big(\HFWBC,\,\|\HBebbes\|_\HFWBC \big)}     
\newcommand\HFWBPCP{{ \HFWsp(\HFBFsp,\HFCsp')}}     
\newcommand{\HFWCOli}{ \WTsp \HFCOsp \HFlisp}     
\def\WTsp#1#2{{\HFWsp(#1,#2)}}     
\newcommand{\HFlisp}{{\HFlsp^1}}     
\newcommand{\HFWCOliG}{ \WTsp \HFCOsp \HFlisp (G) }     
\newcommand{\HFWCOliRd}{\WTsp \HFCOsp \HFlisp (\HFRdst) }     
\newcommand{\HFWCOliRdN}{{\big( \HFWCOliRd, \, \|\HBebbes\|_\HFWsp \big)}}     
\newcommand{\HFWCOlisp}{ \WTsp \HFCOsp \HFlisp}     
\newcommand{\HFWCOlp}{{\WTsp \HFCOsp \HFlpsp }}     
\newcommand{\HFlpsp}{{\HFlsp^p}}     
\newcommand\HFWCOlpRd{{\HFWsp(\HFCOsp,\HFlpsp)(\HFRdst)}}     
\newcommand{\HFWCOlt}{{\WTsp \HFCOsp \HFltsp }}     
\newcommand{\HFltsp}{{\HFlsp^2}}     
\newcommand{\HFWCOltRd}{\WTsp \HFCOsp \HFltsp (\HFRdst) }     
\newcommand\HFWFLili{{\HFWsp(\HFFLi \HFnnth,\HFlisp)}}     
\newcommand\HFWFLilin{{\HFWsp(\HFFLi,\HFlinsp)}}     
\newcommand{\HFlinsp}{{\HFlsp^\infty}}     
\newcommand\HFWFLilt{\HFWsp(\HFFLi,\HFltsp)}     
\newcommand\HFWFLinlin{ \HFWsp( \HFFLinsp, \HFlinsp)}     
\newcommand\HFWFLinlinRd{{\HFWsp(\HFFLinf,\HFlinsp)(\HFRdst)}}     
\newcommand\HFWFLplp{{\HFWsp(\HFFT\HFLsp^p,\HFlpsp)}}     
\newcommand{\HFWFLplq}{{\HFWsp(\HFFT\HFLsp^p,\HFlqsp)}}     
\newcommand\HFlqsp{{ \HFlsp^q}}     
\newcommand\HFWFlinlin{ \HFWsp( \HFFLinsp, \ell^\infty)}
\newcommand\HFWLilin{{\HFWsp(\Lisp,\HFlinsp)}}     
\newcommand\HFWLilinf{{\HFWsp(\Lisp,\HFlinsp)}}     
\newcommand\HFWLilp{{\HFWsp(\Lisp,\HFlpsp)}}     
\newcommand\HFWLinfli{{ \HFWsp(\HFLinsp,\HFlisp) }}     
\newcommand\HFWLpli{{ \HFWsp(\HFLpsp,\ell^1) }}     
\newcommand\HFWLplq{{ \HFWsp(\HFLpsp,\HFlqsp) }} 
\newcommand\HFWLplqRd{{ \HFWsp(\HFLpsp,\HFlqsp)(\HFRdst) }}     
\newcommand\HFWLtli{\HFWsp(\HFLtsp,\HFlisp)}     
\newcommand\HFWLtliRd{{\HFWsp(\HFLtsp,\HFlisp)(\HFRdst)}}     
\newcommand\HFWLtlin{\HFWsp(\HFLtsp,\HFlinsp)}     
\newcommand\HFWLtlt{{\HFWsp(\HFLtsp,\HFltsp)}}     
\newcommand\HFWMlin{ \HFWsp(\HFMsp,\HFlinsp)}     
\newcommand\HFWMlinf{ \HFWsp(\HFMsp,\HFlinsp)}     
\newcommand\HFWMlinfRd{ \HFWsp(\HFMsp,\HFlinsp)(\HFRdst)}     
\def\WTsp#1#2{{\HFWsp(#1,#2)}}     
\newcommand{\HFYspN}{(\HFYsp, \, \|\HBebbes\|_\HFYsp)}     
\newcommand{\HFZdst}{{\HFZst^d}}     
\newcommand{\HFZst}{{\mathbb Z}}     
\newcommand{\HFZtdst}{{\HFZst^{2d}}}     
\newcommand\HFboxcar{{{\bf 1}_{[-1/2,1/2]}}}     
\newcommand{\HFcheckm}{{^\checkmark}}     
\newcommand\HFclam{{c_\lambda}}     
\newcommand{\HFellGTr}{{ (\HFlisp,\HFltsp \negthinspace , \HFlinsp) }}     
\newcommand\HFfchk{{f^\checkmark}}     
\newcommand{\HFgd}{{\widetilde{g}}} 
\newcommand\HFhatf{{\widehat{f}}}     
\newcommand{\HFhatg}{{\widehat{g}}}     
\newcommand\HFhatsi{{\widehat{\sigma}}}     
\newcommand\HFhkr{\hookrightarrow}     
\newcommand{\HFintRd}{\int_{\HFRst^d}}     
\newcommand\HFjapy{{\langle y \rangle}}     
\newcommand{\HFkiZd}{{{k \in \HFZdst}}}
\newcommand\HFliLam{{\HFlisp(\Lambda)}}     
\newcommand{\HFliZ}{{\HFlisp(\HFZst)}}     
\newcommand{\HFlinfnorm}[1]{{\lVert #1 \rVert_{\infty}}}     
\newcommand{\HFlinorm}[1]{{\lVert #1 \rVert_1}}     
\newcommand{\HFlsp}{{\boldsymbol\ell}}     
\newcommand\HFlqwsp{{\HFlsp^q_w}}     
\newcommand\HFltLam{{ \HFltsp(\Lambda)}}     
\newcommand{\HFltwsp}{{\HFlsp^2_{w}}}     
\newcommand{\HFprojcvh}{{ \widehat{\circledast} }}     
\newcommand\HFqandq{{ \quad \mbox{and} \quad }}     
\newcommand\HFsigSOP{{ \sigma \in \HFSOPsp }}     
\newcommand{\HFspec}{\operatorname{spec}}     
\newcommand{\HFsumkZd}{\sum_{\HFkiZd}}     
\newcommand{\HFsumlaLa}{\sum_{\lambda\in\Lambda}}     
\newcommand{\HFsupnorm}[1]{{\lVert #1 \rVert_\infty}}     
\newcommand{\HFsupp}{\operatorname{supp}}     
\newcommand\HFsuth{{ \, | \, } }     
\newcommand{\HFwphi}{\widetilde{\varphi}}     
\newcommand\HFwstd{{w^* \negthinspace -}}     
\newcommand\HFxRd{{x \in \HFRdst}}     

\def\HFCoosp{{\boldsymbol{C \HFnth o}}}
\def\HFPgLam{{\operatorname{P}_{\HFnth g, \Lambda}}}
\def\HFMbnorm#1{{\|#1\|_\HFMbsp}}
\def\HFnormta#1#2{{  \| {#1}   \|_{#2} \, }}
\def\HFBnorm#1{{ \| #1 \|_\HFBsp }}

\makeindex             

\begin{document}

\def\HFcG{{G}}

 \title*{The Concept of Wiener Amalgam Spaces}

\author{Hans G. Feichtinger orcidID{0000-0002-9927-07423}}
\institute{Hans G. Feichtinger \at Faculty of Mathematics, University of Vienna, Oskar-Morgenstern-Platz 1, 1090 Vienna,\newline  and ARI (OEAW, AUSTRIAN Academy of Sciences);  \email{hans.feichtinger@univie.ac.at} }
%
%
\maketitle

\abstract{
This article concerns Wiener amalgam spaces 
and provides some hints about their usefulness in various branches of Harmonic Analysis. Despite the fact that the underlying construction principles
are quite easy to understand and basic facts follow naturally by simple rules, these spaces have not obtained the same popularity as certain other function spaces which are much more complicated to describe and often just serve a very particular purpose.
\newline  \indent
This situation has motivated the author to provide here a summary of the foundations of the theory of Wiener amalgam spaces (and the motivation behind their construction) and a selection of relevant applications, some 45 years
years after the key paper published in 1983.
\newline \indent
We recall first that the so-called {\it classical Wiener amalgam spaces} using local $\HFLpsp$-norms combined with a global $\HFlqsp$-behaviour are already quite
useful, e.g.\ for an improvement of the Hausdorff-Young Theorem with some interesting consequences for Sobolev algebras.  However,  the main emphasis
will be  based on the idea of allowing more general local components  (describing for example smoothness or membership in the Fourier algebra).
This opened the door to the introduction of {\it modulation spaces}, which are now recognized as standard tools in time-frequency analysis.
\newline \indent
We will demonstrate in this article how Wiener amalgam spaces methods
can be used to prove the Sobolev embedding theorem or determine the
pointwise multipliers of Sobolev algebras. We also demonstrate that
the space of multipliers from the classical Wiener algebra $\HFWCOliRd$
into its dual can be identified with $\HFSOPRd$, the space of mild
distributions.
 }


\section{The Classical Wiener Amalgam Spaces}

One of the key arguments for introducing Wiener amalgam spaces (or Wiener-type spaces as they were called in \cite{fe83} and \cite{fe81-1}) is the observation, that there are no inclusion relations between the Banach spaces
$\HFLpRdN$ with $1\leq p \leq \infty$, because of either local counter-examples
or step functions implying non-inclusion results. One way out is to decompose a measurable function on  $\HFRdst$ into blocks of the form $k + Q$ (where $Q$ is the unit cube, or just some relatively compact fundamental domain for $\HFZdst \lhd \HFRdst$), define a discrete sequence (indexed by $\HFZdst$) of the form
$( \HFnormta{f \cdot {\bf 1}_{k+Q}} {\HFLpsp})$ and require that this sequence
belongs to the sequence space $\HFlqsp(\HFZdst)$. These spaces are well-defined Banach spaces with respect to their naturally associated norms.
We will denote by $\HFWLplqRd$. In fact, one can take any lattice  $\Lambda \lhd \HFRdst$ and any relatively compact set containing a fundamental domain for $\Lambda$ instead of just $\HFZdst$,  and one will get the same space with an  equivalent norm.  The classical reference \cite{fost85} provides
a summary of the results known at that time.  The consequences for the
study of product-convolution operators, but also convolution relations
have been given in \cite{busm81}, for example.
Early variants of the Hausdorff-Young theorem for amalgams are also found in \cite{bedu77}, \cite{bedu79}, \cite{st79}, \cite{sz80} or \cite{fo83}.


One can say that Wiener amalgam spaces allow to capture the global
behaviour (in the sense of $\ell^q$-summability or a weighted version
thereof, as described in \cite{he03}) of a local property, namely the $\HFLpsp$-norm over the cubes.
Obviously there is a lot of room for generalization, e.g., by replacing
the local norm by some another Banach function space norm (in the sense
of Luxemburg-Zaanen, see \cite{besh88}), such as Lorentz space or Orlicz space norms for the local (but also the global) component, without changing much
(for some recent references in this direction see \cite{aroz23}, \cite{arozte24},\cite{liliwa24}).

The terminology {\it Wiener} amalgam spaces is justified because the first appearance of such spaces is in the work of N.~Wiener on Tauberian Theorems (\cite{wi32}, and his book \cite{wi33}). Later on such spaces appeared in the work of F.~Holland (\cite{ho75}) and H.E.~Krogstad (\cite{kr76}). The spaces $\HFWLpli$ and in particular $\HFWCOli$ are so-called {\it Segal algebras} in the
sense of H.~Reiter (\cite{re68},\cite{rest00}).

The observation that $\HFWCOli$ is the smallest
among all the spaces $\HFWLplq$ (in fact, it is the closure of the test
functions in $\HFWLinfli$) was observed in \cite{fe77-3}. This characterization
motivated the construction of $\HFWFLili$ as given in \cite{fe81-2}.
In the time-frequency literature the space $\HFWCOli$ is known as
{\it Wiener's Algebra}, not to be confused with the Wiener Algebra
$\HFATN$ of absolutely convergent Fourier series. Note, however, that
the concept of Wiener amalgam spaces is only meaningful in the case
of non-compact groups, while the concept of absolutely convergent Fourier
series expansions makes sense for compact Abelian groups only. Without
going into details and listing all the nice (functorial) properties of
the Segal algebra $\HFSOsp$, which has been introduced as the (generalized)
Wiener amalgam space $\HFWFLili$, one can say that it is a combination of these
two ideas. Locally the elements of $\HFWFLili$ can be characterized as elements
whose localized pieces have absolutely convergent Fourier series expansions via periodic extension, but this local norm is turned into a global norm as a  global $\HFlisp$-amalgam.
The first appearance of $\HFWFLili$ was in a report by J.P.~Bertrandias.

More generally, one can form the Wiener amalgam space $\HFWFLplq$  or weighted version thereof, and to define function spaces similar to Besov spaces by
taking inverse Fourier transforms of such spaces. The corresponding family
of Banach spaces of distributions have been baptised
{\it modulation spaces}  by this author. The key paper is the original
technical report  \cite{fe83-4} which makes use of the continuous description and was published only 20 years later in \cite{fe03}.


\section{BUPUs and Continuous Norm Descriptions}

The definition of Wiener amalgam spaces, especially in the context of LCA
(locally compact Abelian) groups, beyond the Euclidean case, seems to require
the existence of a discrete, cocompact lattice, which unfortunately may not
exists in this generality. Moreover, even in the Euclidean setting one may
ask (for good reasons) what happens of one replaces the given lattice by another lattice (with a corresponding change of the fundamental domain).
In fact, even, the choice of the (relatively compact) fundamental domain for a given lattice might have an influence. Fortunately it is not much more than
a good exercise to verify that all these minor variations of the approach
define the same space with equivalent norms.

The study of Wiener's algebra in \cite{fe77-3} gives an example how one
can define a strictly translation invariant norm on $\HFWCOli$,
that is, a norm satisfying
$$  \HFBnorm{T_x f} = \HFBnorm{f}, \quad x \in \HFcG. $$
To achieve this property Reiter replaces the simple and natural norm
$$    \HFnormta  {f} {\HFWLinfli} := \HFsumkZd \HFsupnorm{f \cdot {\bf 1}_{k+Q}} $$
by the equivalent norm (see \cite{re68}, p.12)
$$  \HFnormta {f} {\HFWsp} :=  \sup_{y \in \HFRdst} \HFnormta  {T_y f} {\HFWLinfli}, $$
which is a bit cumbersome. There are essentially two alternative ways
to obtain a strictly translation invariant norm. The first method is
via atomic decompositions, as described in \cite{fe77-3}. Unfortunately
this method works only for local norms which are strictly translation
invariant and the (unweighted) global component $\HFlisp$.

In contrast, the use of a {\it continuous control function} of the form
$$  \kappa(f): y \mapsto  \HFsupnorm{ f \cdot T_y \phi}, $$
for some compactly supported function $\phi \in \HFLinf$ allows us to
replace the discrete summability by a continuous integral. In the
case of Wiener's algebra this means that an equivalent norm on
$\HFWCOliRd$ can be obtained by the expression
$$  \HFnormta {f} \HFWsp := \HFintRd  \kappa(f)(y) dy
= \HFnormta {\kappa(f)} {\HFLiRd}. $$

This change from a partition of unity - here obtained via the
decomposition of $\HFRdst$ into translations along $\HFZdst$ of
the indicator function of a fundamental domain  -  is one of the
key ideas in the paper \cite{fe83}, which can be seen as the
basis for the {\it modern theory of Wiener amalgam spaces}.

They key results of the frequently cited  paper \cite{fe03}
are: 
\begin{enumerate}
  \item Equivalence of discrete norms (defined via BUPUs) and the
  continuous norm, cf. above;
  \item Opening the concept of amalgamation to much more general
     local as well as global norms;
  \item Establishing pointwise multiplier and convolution results
  for these (generalized) Wiener amalgam spaces.
\end{enumerate}

Without repeating the technical details required in order to reach the
full generality of the concept of Wiener amalgams let us summarize
the key ideas involved in the process.

Let us start with the different ingredients required in order to define
these spaces, namely a partition of unity, a local component and a global component. Overall the norm of a Wiener amalgam space $\HFWBC$ allows us to
capture the {\it global behaviour}  (using the solid, translation invariant
BF-space $\HFCspN$) of a local norm, expressed by the norm of some {\it localizable} space $\HFBspN$. We postpone the easy extension of a global
$\HFlqsp$-behaviour to weighted variants to a later section. What is important
to understand the concept of BUPUs (bounded uniform partitions of unity)
is the meaning of the different words:
\begin{enumerate}
  \item {\it Boundedness} of a family $(\phi_i)_{i \in I}$
    of localizing functions (either constituting
   a discrete partition of unity, or creating the continuous control function)  refers to {\it bounded action on} $\HFBspN$: For some constant
   $C > 0$ one has:
   $$   \HFBnorm {f \cdot \phi_i} \leq   C \HFBnorm{f},
   \quad i \in I, f \in \HFBsp.$$

  \item The term uniform refers to the {\it uniform size} of the members of the partition, one could say the diameter of the support of all the
      functions used by the partition of unity is uniformly controlled.
      We write that $|\Psi| \leq \delta$ if $\HFsupp(\psi_i) \subseteq B_\delta(x_i)$ for all $i \in I$ and suitable points $x_i \in \HFRdst$.
\end{enumerate}

Of course the choice of such a partition of unity depends on the function
space $\HFBspN$ used for the local measurements, but generally speaking it will
be a smoothness condition of some form. If we want to measure the local
pieces in some Sobolev norm (or any other classical function space in the
sense of H.~Triebel) we could just assume a sufficiently high order or
smoothness. Since the classical B-splines form such BUPUs they are good examples of this type.

However, in some cases  smoothness in the classical
sense is not necessary, and  it may be enough to have boundedness in
the Fourier algebra $\HFFLiRdN$. In fact, the well known convolution
relation $\Lisp \ast \HFLpsp \subseteq \HFLpsp$ implies a pointwise relationship
$ \HFFLisp \cdot \HFFLpsp \subseteq \HFFLpsp$, hence for $\HFBsp = \HFFLpsp$ it
suffices to assume that the partition is bounded in the Fourier algebra
sense, with $\HFnormta {f} {\HFFLisp} = \HFnormta {\HFhatf} \Lisp$.

The first and natural choice for BUPUs are the so-called
 {\it regular} BUPU, of the form $(T_\lambda \phi)_{\HFlainLa}$,
 where $\Lambda \lhd \HFRdst$ is some lattice in the group $\HFRdst$.
Again, the classical B-splines, for example the basis for the
space of {\it cubic spline functions} on $\HFRst$, are a good example.

There are however good arguments to for taking more general approach (as done
in \cite{fe03}), keeping the boundedness and uniformity described above
in mind, but allowing much more freedom in the choice of the individual
pieces. This is first of all important in the context of irregular
sampling, where sampling points are given and the so-called Voronoi
domains (of sets defined by the nearest neighborhood relationship)
have to adapt to the given set of sampling points. In this context
the notion of {\it relatively separated} and {\it well-spread} point
sets $X = (x_i)_{i \in I}$ arises. We do not go deeper in this direction,
but mention that in such a case one has to control the overlap between
the neighboring members of an (irregular) partition of unity. In other
words, in the most general version a BUPU
$\HFPsifam$  should satisfy:  For some constant $C_\Psi > 0$ one has
$ \sup_{i \in I} \# \{ j \HFsuth \psi_j \cdot \psi_i \neq 0 \} < C_\Psi. $

Such a condition is relevant whenever it is not possible to guarantee the
existence of a regular BUPU, but more importantly whenever it comes to
discuss {\it joint refinements} for BUPUs (in the spirit of Riemannian sums),
see for example \cite{fe22}.

Later on it turned out that such a finite overlap condition is the only
constraint that has to be preserved in the even more general context
of {\it decomposition spaces} as described in \cite{fegr85} and \cite{fe87},
see the next section about duality and multipliers.


\section{Duality and Interpolation}

The family of Wiener amalgam spaces has a number of good properties. Among others it is (essentially) closed under duality and complex interpolation,
and inclusion can be established coordinatewise.
This allows us to obtain interesting results by verifying the details only
for limiting cases and then apply interpolation arguments. We will demonstrate
this by recalling two versions of the Hausdorff-Young Theorem in the context
of Wiener amalgam spaces.

But first let us recall the characterization of dual spaces for Wiener amalgams. While this is no problem for the classical case, where a given
measurable function is split into pieces with disjoint support so that
a direct sum decomposition is available, one has to be a bit more careful
for the case of general (smooth) BUPUs. This is the reason why the corresponding result is a bit hidden in the literature. Actually it appears
in an even more general context in \cite{fegr85}, where it becomes clear
that the finite overlap condition on the support of the members of the BUPU
play an important role.

%
While the general result (see Theorem 2.8 in \cite{fegr85})
might be not so easy
to digest, we present an simple variant in the context of tempered distributions.
\begin{theorem}\label{WBCduality}
Given a localizable Banach space $\HFBspN$  and a solid translation invariant Banach space $\HFCspN$ of tempered distributions such that $\HFScRd$ is dense
in both $\HFBspN$ and $\HFCspN$. Then their dual spaces are also Banach spaces
of tempered distributions and $\HFWsp(\HFBsp',\HFCsp')$ is well defined. Moreover, $\HFScRd$ is dense in $\HFWBCN$ and we have
$$  \HFWBC' = \HFWBPCP. $$
\end{theorem}
\begin{proof}
We just give an outline of the key arguments for this basic fact.
Obviously the setting of BUPUs guarantees the limited overlap condition
required in the abstract setting of {\it decomposition spaces}. The auxiliary
systems $\psi_i^* := \sum_{\psi_j \cdot \psi_i \neq 0} \psi_j $
are thus also bounded and satisfy $ \psi_i = \psi_i \cdot \psi_i^*$
for any $i \in I$. In fact, they allow (via the same discrete solid
sequence space norm) to give an equivalent norm based on
$\HFBnorm{f \cdot \psi_i^*}$. Thus it is easy to verify that
any element of $\HFWBPCP$ defines a bounded linear functional on $\HFWBCN$.

For the converse a so-called pigeonhole  principle can be applied. One
starts by showing that the index set $I$ of the BUPU can be split into
finitely many subsets $I_l, 1 \leq l \leq L$, such that one has (for example)
$ \psi_r^* \cdot \psi_r^* = 0$ for $r, s \in I_l$ and $r \neq s$.
Assume now that an abstract functional on $\HFWBCN$ is given, which of
course will be represented locally an element $\sigma \in \HFBFsp$. If
$\sigma \notin \HFWBPCP$ then also one of the finitely many subseries
$ \sum_{i \in I_l} \sigma \psi_i^* $ (for some $l$)
would not belong to $\HFWBPCP$, which can be shown to lead to a contradiction.
\end{proof}
\begin{remark}
The above argument indicated that the disadvantage arising from the use
of smooth partitions of unity, which necessarily imply some overlapping
between the members of the partition, makes things more complicated, but
the control of the overlap allows to still derive the expected results,
meaning here that duality can be obtained in the coordinate-wise sense.
\end{remark}

The key result concerning pointwise multiplier results follow the
same principle, see Cor.2.4 (p.107)  and Thm.2.11 (p.110) of \cite{fegr85}. Results
about (complex) interpolation are derived already in \cite{fe81-1}
for Wiener amalgams specifically, but the same methods apply in the
context of decomposition spaces, using the idea of retracts to
sequence spaces consisting of function spaces, which is  a
method used  in recent papers, see e.g.\ \cite{suyayu24}.

%

The connection between continuous covering systems and discrete (admissible = with finite overlap condition) is studied in  \cite{fe87}. Starting from a {\it continuous covering} it describes a selection method that leads to equivalent discrete coverings, generalizing the norm-equivalence of continuous/selective
norms (via control functions) and discrete norms of Wiener amalgam spaces (cf. \cite{fe83}) using BUPUs. It should also be noted that Theorem 4.2 of \cite{fe87} provides a construction of a BAPU (admissible partition of unity)
in such a situation which only makes use of the finite overlap condition
combined with the algebra property of the algebra of pointwise multipliers
of $\HFBspN$, but {\it not} a pointwise inversion result that is typically used
in the construction of partitions of unity (e.g., in the context of smooth
manifolds).

\section{Convolution of Bounded Measures}

The use of BUPUs goes far beyond the possibility of defining Wiener amalgam spaces. A summary of the various applications is given in \cite{fe24-2}. This
paper also contains various applications, including the construction of
Wiener amalgam spaces and modulation spaces, which are more or less Wiener
amalgam spaces on the Fourier transform side, with local $\HFFLpsp$-behaviour.

In this short section we just want to point out that implicitly the
approach to the convolution of bounded measures is based on ideas closely
connected with the method of convolution of Wiener amalgams.  In the paper
\cite{fe17}  novel mathematical approach to the theory of translation invariant linear systems is presented, which is then the basis for the results in  \cite{fe22}.

The paper \cite{fe17} presents an approach to the convolution of bounded measures without using measure theory, using the following  key ideas:
\begin{enumerate}
  \item Define $\HFMbRdN$ as the dual space to $\HFCORdN$, the Banach algebra
  of all continuous, complex-valued functions which vanish at infinity,
  with the natural $\sup$-norm.
  \item In this way the algebra structure of $\HFCORd$ can be transferred
  to its dual space by adjoint action, or in a more concrete fashion
   by setting $ \mu \cdot h(f) = \mu(h \cdot f)$ for $\mu \in \HFMbRd$ and
   $h,f \in \HFCORd$.
  \item An important step is then to identify $\HFMbRd$ with $\HFWsp(\HFMbsp,\HFlisp)(\HFRdst)$, in fact, by showing that for any BUPU
      $\HFPsifam$ one has
      $$  \sum_{i \in I} \HFMbnorm {\mu \cdot \psi_i} = \HFMbnorm {\mu},
      \quad \mu \in \HFMbRd.$$
  \item Using this fact one can show that there is a natural identification
  of the space of bounded linear operators commuting with translations
  and $\HFMbRdN$, where $\mu \in \HFMbRd$ corresponds to the (external)
  convolution operator on $\HFCORd$, given pointwise by
  $$  C_\mu(f)(x) = \mu \ast f (x) =  \mu(T_x \HFfchk), \quad f \in \HFCORd, $$
  with $\HFfchk(x) = f(-x)$. It is even isometric with respect to the
  natural norms (functional norm for $\mu$ and operator norm for $C_\mu$).
  \item Since these operators obviously form a Banach algebra one can transfer the algebra structure from the operators back to the
      generating bounded measures.
  \item By showing that $\delta_x$, with $\delta_x(f) = f(x)$,  corresponds
  to the translation operator $T_x f = \delta_x \ast f$
  and by taking appropriate
  limits one can derive that convolution is not only associative but also
  commutative.
\end{enumerate}

\section{Sobolev Algebras}

Sobolev spaces have been introduced in analysis in order to
describe the smoothness of functions. One can start with concepts such
as {\it continuity} or {\it differentiability}.
The most natural way to describe increasing smoothness is the
number of continuous derivatives a functions has. For example,
we call a function $f(x)$ on $\HFRst$ a $C^{(k)}$-function
for some $k \in \HFNst$ if $f,f',f''$ etc. up to the
derivative of order $k$, written as $f^{(k)}$ are all well-defined and
continuous functions.

Let us look at B-splines, the convolution powers of the boxcar
functions $\HFboxcar$:  we find that the B-spline of order $2$ (or
degree $1$), i.e.\ the triangular function $\Delta$, is continuous
but not everywhere differentiable, hence it is in $\HFCcRd$ resp.\
$C^{(0)}(\HFRst)$. The next one is continuously differentiable
(but only first order), and locally glued together by quadratic
polynomials. Finally cubic splines are constituted by locally
cubic polynomials (degree $3$, order $4$), which have continuous
derivatives up to order $2$ (continuous {\it curvature}), so
they constitute $C^{(2)}$-functions and are thus very suitable
to plot curves (with continuously changing curvature).

For $\HFLtRd$ one can develop a similar (pointwise) concept, by
assuming that the derivative exists almost everywhere and that
the class constituted in this way also belongs to $\HFLtR$,
and so on (by induction). A more elegant version makes use
of the theory of tempered distributions in the sense of
L.~Schwartz. Finally Fourier methods allow a more flexible  description of
smoothness at a fine scale.

Looking at the behaviour of the derivative operation
(which is of course a [unbounded] translation invariant
operator) we expect that it is a multiplication operator
on the Fourier transform side.  Using the exponential law it is
easy to compute the derivative of $x \mapsto \chi_s(x) = exp(2 \pi i s x)$,
which is $\chi'_s(x) =  2 \pi i s \chi_s(x)$. Hence it is not surprising that the standard formula (for our normalization of the FT) is:
\begin{lemma} \label{differentiation}
Assume that $f \in \HFLiRd \cap \HFCORd$ has a continuous,
integrable derivative (that is, $f$ is an {\it absolutely
continuous functions} in the classical sense), then we have
\begin{equation}\label{derivFT01}
 { \HFFT(f')(s) = 2 \pi \, i\,  s \,  \HFhatf(s)}, \quad s \in \HFRst.
\end{equation}
\end{lemma}

In the context of $\HFLtRN$ one can describe therefore the
number of derivatives by the square integrability of
the corresponding derivatives in $\HFLtR$, and thus
by the membership of $\HFhatf$ in some corresponding
weighted $\HFLtsp$-space on $\HFRst$. This is how {\it Sobolev
spaces}
can be introduced via the Fourier transform
as weighted $\HFLtsp$-spaces. This approach even allows us to
talk about  fractional differentiability.

For $s \geq 0$ the usual Sobolev spaces $\HFHsRdN$ are defined with the help of
 polynomial weights of the form
$v_s(y) = (1+|y|^2)^{s/2}$.
%
\begin{definition} \label{HsRddef}
\begin{equation} \label{HSRddef00}
\HFHsRd := \{f \HFsuth \HFhatf \cdot v_s \in \HFLtRd \}
\end{equation}
with the norm $ \HFnormta f \HFHsRd :=
\HFnormta {\HFhatf v_s} \HFLtRd$.
\end{definition}
In other words, Sobolev spaces are images (under the inverse Fourier
transform) of the weighted spaces $\HFLtvsRd := \{ f \HFsuth f v_s \in \HFLtRd \}$.

For the derivation of some important results concerning results Sobolev spaces we make use of the fact that $v_s$ (for $s \geq 0$) is weakly subadditive
weight function, see \cite{br75}, \cite{fe79} and \cite{gr07} for a discussion of different types of weight functions. We have
\begin{definition}  \label{WSA-def}
A weight function $w$ is called {\it weakly subaddititve} (WSA) if  there
exists some $C>0$ such that
\begin{equation} \label{WSA-def1}
 w(x+y) \leq  C \cdot [ w(x) + w(y)]   \quad \forall x,y \in \HFRdst.
\end{equation}
\end{definition}
 For WSA weights $w$ we have the following result:
\begin{proposition} \label{WSAconv}
For any WSA weight $w$ we have:
 $\HFBsp := \HFLiRd \cap \HFLpwRd$
with natural norm $\HFBnorm{f} = \Lino{f} +
\HFnormta {fw} {\HFLpsp}$
is a Banach algebra with respect to convolution, for any $p \in [1,\infty]$.
\end{proposition}
\begin{proof}
The crucial {\it pointwise estimate} is the following one:
\begin{equation} \label{WSAestim02}
  |[f \ast g] \cdot w|(x) \leq C( |f| \ast |gw| + |fw| \ast |g|)(x),
  \quad \HFxRd.
\end{equation}
In fact, we just have to split the weight factor and use that
$$ w(x) \leq  C \left( w(x-y) + w(y)\right ), \quad x,y \in \HFRdst.$$
Since the estimate
$$ \HFnormta{f\ast g} \Lisp \leq \HFnormta{f}  \Lisp \HFnormta{g} \Lisp $$
is obvious,   the key estimate  follows from (\ref{WSAestim02}) giving us
$$ \HFnormta{ [f \ast g]w} \HFLpsp \leq
C \left(\Lino{f}\HFnormta {gw}{\HFLpsp} +\Lino{g}\HFnormta {fw}\HFLpsp \right).$$
Combining these two estimates we have
\begin{equation}\label{LiLwpest}
  \HFBnorm{f\ast g} \leq C \HFBnorm{f} \HFBnorm{g}, \quad f,g \in \HFBsp.
\end{equation}
Upon replacing the natural norm by $C \HFBnorm{f}$ the constant $C$ can
be assumed to be equal to $1$.
\end{proof}

\begin{corollary} \label{LtwL1}
Assume that $w$ is a WSA-weight. Then
$\HFLtwRdN$ is a Banach algebra with respect to convolution
if  $1/w \in \HFLtRd$. In this case Sobolev's embedding principle
applies, i.e.\ $\HFHsRdN \HFhkr \HFCORdN$.
\end{corollary}
\begin{proof}
This is a simple consequence of the fact that
$ f =  (fw) \cdot 1/w \in \HFLtsp \cdot \HFLtsp \subseteq \Lisp $
due to the Cauchy-Schwartz inequality. Moreover, by the same
argument one can show that the natural norm of $\Lisp \cap \HFLtwsp$
is equivalent to the natural norm on $\HFLtwsp$.

Applying the Fourier transform on both sides we obtain that $\HFHsRdN\hookrightarrow \HFFLiRdN
\HFhkr \HFCORdN$, with the natural norm estimates.
\end{proof}

\begin{remark}
Note that the above argument shows that $\HFHsRd$ is not only continuously
embedded into $\HFLtRdN$ (according to Plancherel, since $\HFLtwRd \hookrightarrow
\HFLtRd$), but also in $\HFCORdN$, hence into $\HFLtsp \cap \HFCOsp (\HFRdst)$ with the
natural norm (sum of the two norms). However, using the Hausdorff-Young
Theorem for Wiener amalgam spaces we can show  $\HFHsRd \hookrightarrow \HFWCOltRd$,
by the argument:   $\HFLtwRd = \HFWsp(\HFLtsp,\HFltwsp) = \HFWsp(\HFFT \HFLtsp, \HFltwsp)
\hookrightarrow \HFWsp(\HFFT \HFLtsp, \HFlisp)$ (using again Cauchy Schwartz in the
last inclusion), which by a variant of the Hausdorff-Young inequality for
Wiener amalgam spaces gives $$\HFHsRd = \HFIFT \HFLtwsp \hookrightarrow \HFIFT \HFWsp(\HFFT \HFLtsp, \HFlisp)
\hookrightarrow  \HFWsp(\HFFLisp,\HFltsp)  \hookrightarrow \HFWCOlt(\HFRdst),$$
which is a space strictly contained in $\HFLtsp \cap \HFCOsp(\HFRdst)$.
The reader is referred to \cite{fe90} for $p$-version of similar results.
\end{remark}

Note that $\HFHsRd$ is not only a Banach space with respect to the natural
norm $\|f\|_\HFHs := \| \hat f w\|_2$, but even a Hilbert space,  because the norm
obviously comes from the scalar product
\begin{equation}\label{HS-scalprod}
    \langle f, g \rangle_\HFHs :=  \langle \hat f w ,  \hat g w \rangle_\HFLtsp =
    \HFintRd \hat f(s)  \, \overline{ \hat g(s)} \, \,  w^2(s) ds.
\end{equation}

\begin{corollary} \label{Hs-RKHS}
$\HFHsRd$ is a reproducing kernel Hilbert space, i.e.\ for each $t \in \HFRdst$
the Dirac measure $\delta_x:  f \mapsto f(x)$ is a continuous linear functional
on the Hilbert space $\HFHsRd$. The kernel $K(x,y)$, consisting of functions
$K( \cdot,y) = k_x(y)$ in $\HFHsRd$ with
\begin{equation}\label{Hs-RKHSphi}
    f(x) = \langle f ,  k_x \rangle_\HFHs  \quad \mbox{for all} \quad x \in \HFRdst, f \in \HFHs,
\end{equation}
is obtained  as collections of shifts $T_x \varphi$, with $\varphi = \HFIFT(1/{w^2})$.
\end{corollary}
\begin{proof}
Since $\HFHsRd$ is continuously embedded into $\HFCORdN$ it is clear that the family
of point measures acts   {\it uniformly boundedly} on $\HFHsRd$. Hence we only
have to prove the explicit representation of $\delta_0$ on $\HFHsRd$ and then
(using the definitions) the covariance of the situation: shifting
the point evaluation ($\delta_0$ to $\delta_x$) corresponds to shifting the generator representing $\varphi$. We get the representation using the Fourier inversion formula. Note  that we do not need that $f$ itself is in $\HFLiRd$, since the inverse Fourier transform a priori defined as an $\HFLtsp$-FT has the usual form as integral  if $\hat f \in \HFLiRd$, which is true for any
$f \in \HFHsRd$, if $s > d/2$:
\begin{equation}\label{Hs-delta0rep}
    f(0) =  \HFintRd  \hat f (s) ds =  \HFintRd  {\hat f}(s) w(s) \cdot \frac{1}{w^2}(s) w(s) ds =: \HFintRd \hat f(s)  \, \overline{ \hat \varphi(s)} \, \,  w^2(s) ds,
\end{equation}
if we put $\varphi = \HFIFT(1/{w^2})$.

As a matter of fact we have $\HFjapy^r \in \HFSORd$
if (and only if) $r < -d/2$.

Now we can apply the shift invariance of the scalar product (exercise):
\begin{equation}\label{HS-scalshift}
   \langle T_x f, T_x h \rangle_\HFHs =   \langle f, h \rangle_\HFHs \quad \mbox{for all} \quad
     f, h \in \HFHsRd,
\end{equation}
because we know that translation goes to modulation on the Fourier transform side, but having the same modulations within a scalar products means (due to the fact that one is taking a conjugation in the scalar product) that it is unchanged. Technically speaking one could argue that translation goes into modulation, but for {\it all weights} modulations are unitary on the corresponding weighted $\HFLtsp$-spaces.
\end{proof}
Probably some more details can be found in the paper on robustness (\cite{fewe04}), where it is shown that  the reproducing kernels depend in a continuous way on the weight, e.g., on the parameter $s$ for the case of the weights $w = w_s$ on $\HFRdst$. Again Wiener amalgam methods are used in order
to describe the continuous dependence of minimal norm interpolation operators
from both the degree of smoothness $s$ and the lattice chosen. Since a change
of lattice constants cannot be controlled by an operator norm error estimate
of the corresponding operators  (but only continuous dependence in the strong
operator topology) such estimates are more delicate.

Different choices of weights have been treated in the (recent) PhD thesis of Rosa Aceska, treating the case of {\it variable smoothness} (\cite{ac09}).




\section{Wavelets, Gabor Expansions and Coorbit Theory}

Wiener amalgam spaces over general locally compact groups play
an important role also in the theory of {\it coorbit spaces} as
developed in a series of papers by the author with K.~Gr\"ochenig
in the late 80th, see \cite{fegr88}, \cite{fegr89}, \cite{fegr89-1}
and \cite{gr91}.
This was already indicated in  the highly cited 
survey article \cite{hewa89} which mentions Wiener amalgams as important tools in the context of wavelet theory and Gabor Analysis.

In fact, this is not surprising, since the starting point of {\it coorbit
theory} is the use of a generalized {\it wavelet transform}, named a
{\it voice transform}, defined on a locally compact group $\HFcG$, acting on a
Hilbert space (typically $\HFLtRdN$) by some unitary, irreducible representation.
The key results of {\it coorbit space theory} describe the following scenario:
\begin{enumerate}
  \item Using the continuous {\it voice transform} $f \mapsto \HFVgf$ one obtains an isometric    identification of the Hilbert space $\HFHilb$ with its range under this transform,  a closed, left translation invariant   Banach space of continuous functions in $\HFLtcGN$.
  \item Assuming that one has a good fiducial vector (Gabor atom, mother
  wavelet) one can select from a larger reservoir (say the space
  $\HFScPRd$ of tempered distributions) those elements which belong to
  a given translation-invariant solid BF-space, say a weighted $\HFLpsp$-space
  over the acting group $\HFcG$.
  \item The fact that the identity operator can be described as a convolution
  operator over $\HFcG$ which describes the projection from the generic ambient
  BF-space $\HFYspN$ onto the range of the voice-transform (inside of $\HFYsp$).
  \item Discretization of this reproducing convolution operator allows
  to recover any $f$ from sufficiently dense discrete collections
  of samples $\HFVgf(x_i), i \in I$, and  derive
  {\it atomic decompositions} for the elements of  coorbit spaces
  $\HFCooY$.
\end{enumerate}
For the derivation of all these results Wiener amalgams play an important role. In fact, the choice of ``good windows'' $g \in \HFHilb$ implies that the autocorrelation function $\HFVgg$ is not only integrable, but in fact in
$\HFWCOliG$ or even a weighted version of this space. This in turn implies
good properties of the voice transform, e.g., $\HFVgf \in \HFWCOlt(\HFcG)$ for
$f \in \HFHilb$, due to convolution results for Wiener amalgams. Note that these
convolution relations are very similar to the Euclidean case (which work
 ``componentwise'') for the case of so-called $[IN]$-groups, including the (reduced) Heisenberg group $\HFTFd \times \HFTst$. In contrast, one has to
distinguish between left and right Wiener amalgam spaces on more general groups, such as the ``$ax+b$-group''  of affine transformations of the real line (or $\HFRdst$), which makes estimates more complicated and requires stronger
assumptions on the atoms used (see \cite{fegr89} for technical details). Anyway, let us just mention, that for Gabor Analysis, in particular
for the analysis of {\it regular Gabor families} of the form $\HFFAMglaLa$  consisting of time-frequency shifted version of that Gabor atom $g \in \HFSORd$ (typically), with $\HFVgg \in \HFLiTFd$ (or in fact equivalently $\HFVgg \in \HFWCOli(\HFRtdst)$!) and $\Lambda \lhd \HFTFd$ a discrete lattice of the form
$\Lambda = \AA \ast \HFZtdst$ for some non-singular matrix $\AA$ all the
required estimates can be obtained quite easily using standard results
for Wiener amalgams.

\section{Irregular Sampling and Wiener Amalgam Spaces}

The starting point for the treatment of the so-called {\it irregular
sampling problem} is the classical Shannon Sampling Theorem, which
tells us that a band-limited function $f \in \HFBsp^\Omega \subset \HFLtR$
 (meaning that $\HFspec(f) := \HFsupp(\HFhatf)$ is some given compact set $\Omega$ of the frequency domain) can be recovered from regular samples using the
well-known SINC-function $\HFSINC = \HFIFT({\bf 1}_{[-1/2,1/2]^d})$,
for any sufficiently small value of $\alpha$ (the {\it Nyquist rate}),
in the following form
$$ f(t)  = C_\alpha \HFsumkZd  f(\alpha k) \HFSINC(t-\alpha k),
\quad t \in \HFRdst,$$
for a suitable normalizing factor $C_\alpha > 0$.

Wiener amalgam spaces are also very useful for the study of the
so-called  (principal) {\it shift invariant} spaces,
also called {\it spline-type spaces} by the author
 because the prototypical case are spaces
of splines of any order, such as {\it cubic spline functions} on
the real line\footnote{This author avoids the term ``shift-invariant''
because it is easy to confuse shift-invariance with translation invariance. Furthermore the corresponding spaces depend on both the lattice used and the building block, which could be the cubic B-spline, i.e.\  the convolution product of order $4$ of the standard box-car function ${\bf 1}_{[-1/2,1/2]}$.}.
Such spaces can be characterized by the fact that they have Riesz bases  obtained from a single function and its translates, i.e., a family of the
form $(T_\lambda g)_{\HFlainLa}$. It is not difficult to show that in such a case
the biorthogonal Riesz basis is also generated by some $\HFgd$ in the closed
linear span, or $\HFgd = \HFsumlaLa \HFclam T_\lambda(g)$, with $(\HFclam)_{\HFlainLa}
\in \HFltLam$.

If one assumes that $g \in \HFWLtliRd$, then the sampled convolution products
of the form $ \langle T_\lambda g, g \rangle_\HFHilb = g \ast g\HFcheckm(\lambda)$,
with $\HFlainLa$ and $g \HFcheckm(x) = g(-x)$, determining the quality of Gram-matrix for the given family of translates, belong to $\HFliLam$, because we have
$$  \HFWLtli \ast \HFWLtli \subset \HFWFLili \subset \HFWCOliRd, $$
and consequently the sampling sequence under consideration belongs to $\HFlisp(\Lambda)$. This fact allows to infer (via a classical version
of Wiener's Inversion Theorem) that the biorthogonal Riesz basis is
constituted by translates of an element $\HFgd$  of the closed linear span of the given family, with $\HFlisp$-coefficients.  For example, if $g \in \HFWCOliRd$, then $\HFgd \in \HFWCOliRd$ as well.   The orthogonal projection from
$\HFLtRdN$ onto the closed linear span is thus given by the following
expression
\begin{equation} \label{PgLam02}
 \operatorname{P}_{g,\Lambda}(f) = \HFsumlaLa f \ast g \HFcheckm(\lambda) T_\lambda \HFgd, \quad f \in \HFLtRd,
\end{equation}
with convergence not only in $\HFLtRdN$ but also in $\HFWCOliRdN$, hence
also uniformly.

This situation also has the advantage that $\HFPgLam$ also extends naturally
to an operator on $\HFLpRdN$, for any $p \in [1,\infty]$, with uniform
bounds (a multiple of the norms of $g$, respectively $\HFgd$,  in $\HFWCOliRdN$).


Wiener amalgam spaces have been extremely useful for the mathematical
analysis of iterative schemes for the recovery of band-limited functions
in various function spaces, such as $\HFLpRdN$. The basic strategy
is very often the same: Given the sampling values at a sufficiently
discrete set of points $(x_i)_{i \in I}$ one tries to approximate
$f$ by suitable functions from the same space. This is usually done
with the help of an adapted partition of unity (such as the system
of triangular functions used to create the piecewise linear interpolation
operator from the sampling values of $f$, in the case of the real line),
followed by a projection operator onto the space of band-limited functions,
which in fact may destroy the interpolatory property of the first approximation, but ensures (with a suitable analysis, see \cite{fegr94}) that
the resulting series belongs to the same Banach space where $f$ is taken from,
so that an iterative algorithm can be applied, which is essentially based on a fixed point theorem for Banach spaces.

Among the critical points to mention here is the equivalence of the continuous
$\HFLpsp$-norm with the norm of $\HFWCOlpRd$ on a space of band-limited functions.
This allows us to estimate the approximation error in the first step
of the iterative reconstruction process,  by making use
of the oscillation function (see \cite{fugr07} for extension of these ideas).
The error estimate (jitter error, aliasing error) is also based on considerations involving Wiener amalgam spaces and their convolution relations, see \cite{fegr94}.

It is moreover important to be able to provide convergence rates which depend
only on the band-width of the function which has been sampled, but not on the
choice of the norm, say the $\HFLpsp$-norm, with $1 \leq p \leq \infty$. Since the natural (but naive) multiplication operator with the indicator function
of the spectral domain $\Omega$ (which corresponds to convolution with the
SINC-function in the time domain) requires a restriction to $p \in (1,\infty$)
(because $\HFSINC \in \HFLpR$ only for $1 < p < \infty$), and even then the norms
grow as $p \to 1$ (or to $\infty$) one has to resort to a two-layer process
which is described in all detail e.g.\ in \cite{fegr92-3}: Since there
are now convolution idempotents in $\HFLiRd$ one has to choose some band-limited
function $g \in \HFLiRd$ with $\HFhatg(\omega) \equiv 1$ on $\Omega$ and another
such function $h$ with $ h \ast g = g$. Thus overall one can use the identities
$$   f = f \ast g = f \ast (g \ast h) =  (f\ast h) \ast g $$
for any band-limited function in any of the spaces $\HFLpRd$ for $1 \leq p \leq \infty$. Upon replacing $f \ast h$ by a discrete version of the Radon measure
$f \ast h$ one comes up with an approximation operator whose range is in the
closed linear span of the translates of $g$, which in turn belong to a space
of band-limited functions with a slightly larger spectrum, say
$\Omega' = \Omega + B_r(0)$, for some small value of $r > 0$.

Similar arguments are also used in \cite{fe22-1}, or in the joint papers
with A.~Gumber \cite{fegu21} and \cite{fegu23}.

\section{Spline-type Spaces and Quasi-Interpolation Operators}

It turns out that very similar arguments, also involving Wiener amalgams,
can be performed for the case of so-called spline-type spaces (see \cite{fe02},
or \cite{ro09}), which are usually called (principal) shift-invariant spaces.
Such spaces are typically generate by the translates along some lattice $\Lambda \lhd \HFRdst$ of a given function $\varphi$, typically in $\HFWCOliRd$, which forms a Riesz basic sequence, i.e., a Riesz
basis for its closed linear span. Even for $\varphi \in \HFWLtliRd$ it is
easy to describe this Riesz basis condition, which in turn is equivalent
to the bounded invertibility of the corresponding Gram-matrix. But this
Gram matrix is circulant, with entries generated from the samples of
$\varphi \ast \varphi \HFcheckm$ (flipped version), restricted to $\Lambda$.
With the help of the Hausdorff-Young principle one finds that this
Gram-matrix can be equivalently described by a non-vanishing
condition on the $\Lambda^\perp$-periodized version of $|\widehat{\varphi}|^2$.
But the Hausdorff-Young Theorem for Wiener amalgams implies that
$ \HFFT(\HFWLtliRd) \subset \HFWFLilt$ and thus
$$ |\widehat{\varphi}|^2 \in \HFWFLilt \cdot \HFWFLilt \subset \HFWFLili = \HFSORd $$
by the Cauchy-Schwarz inequality.

\def\HFwphi{\widetilde{\varphi}}

While the shifted $\HFSINC$ functions form a tight frame for the
Hilbert space of band-limited functions in $\HFLtRdN$ and thus the
orthogonal projection operator is simply the convolution by $\HFSINC$,
resp. the operator $f \mapsto \HFIFT({\bf 1}_\Omega \cdot \HFhatf)$,
we need the (unique) {\it biorthogonal Riesz basis} for the case
of spline-type spaces. Again, the key arguments are convolution relations
for Wiener amalgam spaces combined with the classical variant of Wiener's
Inversion Theorem. It states that a function $f$, with a sequence of Fourier coefficients ${\bf c} \in \HFliZ$, that satisfies $f(s) \neq 0$ for any $s \in \HFTst$
has the property that also $s \mapsto 1/f(s)$ is a function with absolutely
convergent Fourier coefficients. Such arguments imply that for any
generator $\varphi \in \HFWCOliRd$ also the biorthogonal family $T_\lambda \HFwphi$, with $\HFlainLa$, is generated by another function $\HFwphi$.
Thus the projection operator  $\HFPgLam$
on such a spline-type space can be described as an operator of the form (we assume for convenience that $\varphi$ and hence $\HFwphi$ are even functions)
$$  \HFPgLam: f \mapsto   \HFsumlaLa  f \ast \HFwphi(\lambda) T_\lambda \varphi. $$
But the standard properties of Wiener amalgam spaces provide us with the
fact that $ f \ast \HFwphi \in \HFLpsp \ast \HFWCOli \subset \HFWLilp \ast \HFWCOli
\subset \HFWCOlp$, and furthermore the convolution of
$ \HFsumlaLa f\ast \HFwphi(\lambda)$ with $\varphi \in \HFWCOli$ is again
convergent in $\HFLpRd$ for $1 \leq p < \infty$, with constants which are independent of $p \in [1,\infty]$.

\section{Reconstruction from Averages and Spline Type Spaces}

The possibility of adapting the iterative reconstruction algorithms for
spaces of band-limited functions to what has been called spline-type function
spaces was observed in two early joint papers with Akram Aldroubi, namely
%
\cite{alfe97}  and
\cite{alfe98}. 
This direction has found big resonance in the community. The survey article
\cite{algr01} has a large number of citations and has become the reference
point for a long list of publications which cannot be mentioned here.

Wiener amalgam spaces also allow to control different types of error
for the irregular sampling problem, be it in the context of band-limited
functions or spline-type spaces. A first systematic exploration of such
a viewpoint is given in the papers \cite{fe92,fe92-3} and the joint papers
with K.~Gr\"ochenig \cite{fegr92,fegr93}.

The results about the control of {\it jitter error}, are a good example
of the power of the method of Wiener amalgams. Results of this type allow
to control the relative error in the reconstruction of say a band-limited
function even if the exact location of the sampling points is not known.
It would be easy to control such an error if the total sum or maybe even
the $\HFltsp$-sum of the error (concerning the location), so to say
the expression $\sum_{i \in I} |x_i - x'_i|^2$ was controlled, but for
applications the more delicate question is: is it enough to know that
$ \sup_{i \in I} |x_i - x'_i| \leq \delta$, where $x_i$ describes the true
(assumed to be valid) and $x'_i$ the actually used sampling positions.

Once one has good uniform control of jitter errors one can go on
and try to reconstruct from local averages, using similar arguments.
Suitable normalized local averages can be viewed as local integrals
over jitter errors and thus often the worst case estimates arising in
the analysis of jitter errors is used in this context.
One of the first papers in this direction was \cite{alfe02}.
A.~Aldroubi has published a couple of papers on this topic, see \cite{al02,alsuta04,alsuta05}.   A.~Aldroubi
has published several papers on this topic, see  \cite{al02,alsuta04,alsuta05}.



\section{Weighted Wiener Amalgam Spaces}

The approach to Wiener amalgams given in \cite{fe83} offers great flexibility
concerning the ingredients to be used compared to the ordinary spaces that have
just local $\HFLpsp$-norms with global $\HFlqsp$-summability. A good early
summary is provided by \cite{he03}.

First of all one can use more general local components locally, even
if one wants to stay with the simple decompositions of measurable
functions into blocks of equal size (supported on $k + Q$, say).
One can use more or less any Banach Function Space (in the sense
of Luxemburg-Zaanen, see \cite{za67},
or using modern terminology as given in
\cite{hosayaya17}, including Lorentz and Orlicz spaces among others.
In the context  of the discussion of compactness on Banach spaces
of functions or distributions and elsewhere the author makes useb
of the notation of {\it solid BF-spaces}, which is almost the same.

Concerning the global behaviour the first step is to use weighted
variants of the sequence spaces, thus making use of spaces of
the form $\HFWsp(\HFLpsp,\HFlqwsp)$. Here one can think of the weight
(on the discrete lattice $\HFZdst$) as being obtained by sampling
a so-called {\it moderate weight function} $w$ (see \cite{fe79} or \cite{gr07}) defined over all of $\HFRdst$ at the lattice points $\HFZdst$.
In this way it can be guaranteed that different descriptions (changing
the lattice or the fundamental domain) still define the same space
with equivalent (natural) norms.

For two moderate weights it is easy to verify that one has
$$ \HFWsp(\HFLsp^p_{\HFnth w_1}, \ell^q_{w_2}) =
 \HFWsp(\HFLpsp \HFnth, \ell^q_{w_1 w_2}), $$
and therefore we recommend choosing the variant on the
right hand side (putting all the weight into the global
component) when it comes to using such spaces.

For negative values of $s$ one can define the corresponding spaces by duality, or equivalently, as inverse Fourier images of the corresponding inversely weighted $\HFLtsp$-spaces in the sense of  tempered distributions.

\section{Wiener Amalgams and Stochastic Analysis}

\def\sigrho{{\sigma_{\negthinspace \rho}}}

Wiener amalgam spaces are not only useful in the context of Time-Frequency
Analysis and Classical Fourier Analysis, but also for a treatment of {\it
wide sense stationary processes}. Although this theory has  been developed in the (natural) context of LCA groups it is also convenient for the Euclidean case  $\HFcG = \HFRdst$.

In the PhD thesis \cite{ho89}, summarized in the paper \cite{feho14},
a theory of {\it generalized stochastic process} (GSPs) has been developed,
based on the properties of $\HFSOG$. GSPs
can be modelled as linear operators  $\rho$ from the Banach algebra $\HFSORdN$
into some abstract Hilbert space $\HFHilb$ (which is defined based on
stochastic considerations). Any such process possesses a Fourier transform
(hence a spectral representation, via the inverse Fourier transform)
and an auto-covariance $\sigrho \in \HFSOPRtd$. {\it Wide sense stationarity}
(and other traditional assumptions concerning GSPs) can be expressed by
properties of $\sigrho$, here by invariance under translation along
the main diagonal $\Delta = \{ (x,x), x \in \HFcG \}$. The autocorrelation
of the spectral process $\widehat{\rho} :=  \rho \circ \HFFT $ is just the
$2d$-dimensional Fourier transform of $\sigrho$, or in short
$ \HFFT(\sigrho) =  \sigma_{\HFFT(\rho)}$. Translation invariance of $\sigrho$
is equivalent to a support condition (concentration on the orthogonal
subgroup to $\Delta$). The positive semi-definiteness then in fact
implies that there is a (non-negative) translation bounded measure
on $\HFcG$ completely characterizing the autocorrelation $\sigrho$ of the
GSP $\rho$.

Connections to Brownian motion were described in
\cite{beoh11}.

\section{Quasi-Crystals and Transformable Measures}

A theory of {\it transformable Radon measures}  has been developed by
L.~Argabright and J.~Gil-de-Lamadrid between 1971 and 1974: see \cite{argi71}
for a survey article and \cite{argi74} 
for a comprehensive discussion of this approach, see also \cite{musz02}.

The (more complicated) definition given in these papers can be rephrased
with the help of Wiener amalgams as follows. We recall that the space $\HFWMlinf$
of {\it translation-bounded} (Radon) {\it measures}, which also appears in  \cite{roth84}, is the dual space of the Wiener algebra $\HFWCOli$. In the context of mild distributions, i.e., of $\HFSOPRd = \HFWFLinlin$,  the set
of {\it transformable measures} can be described 
in the following way: 
\begin{lemma} \label{transfmeas}
$${\mathcal M}_T \HFnth (\HFRdst) = \{ \HFsigSOP \HFsuth \HFhatsi \in \HFWMlinfRd, \sigma \in \HFRaRd. \}. $$
\end{lemma}
Here the membership of $\sigma$ in the class $\HFRaRd$ of Radon measures on $\HFRdst$ corresponds to the assumption that one has in addition to the norm
continuity of $\sigma$ on $\HFSORdN$ (or equivalently, on the dense subspace
$\HFSORd \cap \HFCcRd$) also the following family of estimates:
For any compact set $K \subset \HFRdst$ there exists some constant $C_K > 0$ such that $$ |\sigma(f)| \leq C_K \HFlinfnorm{f},  \quad f \in \HFSORd \cap \HFCcRd. $$

A recent survey article describing the usefulness of translation-bounded measures and mild distributions for the mathematical theory of quasi-crystals is given in
\cite{feriscst24}  (under preparation). It contains among others a proof
of a characterization of all the transformable measures which have transformable Fourier transform in the sense of Argabright and Lamadrid. This
characterization was already claimed soon after the publication of \cite{fe81-2}:
\begin{lemma}\label{doubletransfr}
A Radon measure $\mu$ has a transformable Radon measure $\widehat{\mu}$ if and
only if $\mu \in \HFWMlinf$ (i.e., if it is translation bounded) and has,
in the sense of $\HFSOPRd$,  a Fourier transform which is translation-bounded
as well.
\end{lemma}

\section{Mild Distributions as Multipliers}

We conclude this paper with a result which has not been presented
in this form in the literature. It gives an indication of the fact
that the study of multipliers even between ``ordinary Wiener amalgams''
naturally also leads to the appearance of space of distributions,
in the given case of mild distributions (see \cite{fe19,fe20-1}),
i.e.\ the dual space
of the Segal algebra $\HFSORd = \HFWFLili$, namely $\HFSOPRd = \HFWFLinlinRd$.

\begin{theorem} \label{WWPmult}
The space of multipliers $\HFHRdsp(\HFWCOlisp,\HFWMlin)$ can be naturally identified
with the dual space of $\HFSORdN$, i.e.,  with  $\HFSOPsp = \HFWFlinlin$,  
with equivalence of their natural Banach spaces norms.
\end{theorem}

\begin{proof}
First we observe the obvious chain of inclusions, namely
$$  \HFSOsp \HFhkr \HFWCOli \HFhkr \HFWMlin \HFhkr \HFWFLinlin = \HFSOPsp $$
and thus any operator $T$ from $\HFWCOlisp$ into $\HFWMlin$ which
commutes with translation is obviously also an operator
from $\HFSORd$ into $\HFSOPsp$ and thus by
\cite{fe23-2}, Theorem 6,   a convolution
operator by a uniquely determined element $\sigma \in \HFSOPsp$,
meaning that
$$ Tf(x) = \sigma(T_x \HFfchk), \quad f \in \HFSOsp, x \in \HFRdst. $$

Using the obvious consequence of Plancherel's theorem that
$\HFLtsp \ast \HFLtsp = \HFFLisp$ the convolution relation for Wiener amalgams
gives
$$  \HFWCOlisp \ast \HFWCOlisp \subseteq \HFWLtli \ast \HFWLtli \subseteq
 \HFWFLili $$
and the corresponding norm estimates: for some universal $C > 0$
one has
$$  \HFsupnorm{\sigma \ast (f \ast g)} \leq
 C_1 \HFSOPnorm{\sigma} \HFSOnorm{f \ast g}
\leq C \HFSOPsi  \HFnormta {f} {\HFWCOli} \HFnormta {g} {\HFWCOli}. $$

Let us now look at the converse. Given $\sigma \in \HFSOPsp$
we have to verify that $Tf = \sigma \ast f$  also defines a bounded linear
mapping from $\HFWCOli$ into $\HFWMlin$. So let us fix $\sigma$.
Let us first recall that one has automatically
$$ \HFSOsp \ast \HFSOPsp \subseteq \HFCbsp$$
together with the corresponding norm estimate
(based on the convolution tensor product property for $\HFSORd$)
$$ \HFsupnorm{\sigma \ast f} \leq C_1 \HFSOnorm{f} \HFSOPsi. $$

In fact we have by the minimality property of
$\HFSOsp = \HFWFLili$ (see \cite{fe81-2}):
  $$ \HFWCOli \HFprojcvh \HFWCOli = \HFSOsp. $$

But using the associativity of convolution we have then
that $\sigma \ast f$ defines a convolution operator
from $\HFWCOli$ into $\HFCbRd$ (with the sup-norm):
$$ T: g \mapsto  (\sigma \ast f) \ast g \in \HFCbsp, \quad  g \in \HFWCOli. $$

In this case it is easy to verify (see \cite{fe17} for the
detailed argument) that
$$ \sigma \ast f  \in \HFWCOli'= \HFWMlinf, $$
as we have claimed. The verification of the norm equivalence follows the same lines.
\end{proof}

\begin{remark}
Observing that $\HFWCOli$ is a Segal algebra, hence an essential
module over the Banach convolution algebra $\HFLiRdN$ which has
of course bounded approximate units, we note that we can apply
the Cohen-Hewitt Factorization Theorem which gives
$  \HFWCOli = \Lisp \ast \HFWCOli$. Consequently a multiplier, which
obviously commutes with convolutions and  maps $\HFWCOli$ into $\HFWMlinf$,
has to map $\HFWCOli$ actually into $\Lisp \ast \HFWMlinf \subseteq \HFWLilinf$.
Thus this change from the space $\HFWMlin$ of translation-bounded measures to the
space of locally integrable functions $\HFWLilin$  as target space would not change the result.
\end{remark}

By the same argument, but using
$ \HFWLtli \HFprojcvh \HFWLtli = \HFSOsp, $
we obtain:
\begin{corollary} \label{WLtPmult}
The space of multipliers $\HFHRdsp(\HFWLtli,\HFWLtlin)$ can be naturally
identified with $\HFSOPsp$.
\end{corollary}

\section{Wiener Amalgams and Pseudo-Differential Operators}

\if 2 < 1
\cite{coni08-1}
\cite{coni08}
\cite{coni10}
\cite{coni09-3}
\cite{codetr19}
\cite{conitr18}
\fi

Wiener amalgams (as well as modulation spaces) also play a role in connection with the theory of pseudo-differential operators, in particular
in connection with the Schr\"odinger equation. For a collection
of papers in this direction,  by E.~Cordero and her coauthors see
\cite{coni08-1, coni08, coni10, coni09-3, codetr19, conitr18},
but  also \cite{kato18} or \cite{kikose19}.
For recent contributions see \cite{guzh22} or  \cite{guzh24}.

The use of modulation spaces in the field of pseudo-differential
operators would justify at least another survey article. In fact,
by now modulation spaces have become a standard tool there, starting
with the paper \cite{grhe99}, see also \cite{grhe04} or \cite{begrheok05}.
We refer again to the books  \cite{gr01}, \cite{coro20} and \cite{beok20}.
Let us also mention that Wiener amalgam methods have been crucial
in identifying certain exotic coorbit spaces in
\cite{fepipr22}.

\section{The Banach Gelfand Triple}

Among the unweighted Wiener amalgam spaces there are a few spaces
(aside from $\HFLtRd = \HFWLtlt(\HFRdst)$)
which are invariant under the Fourier transform, namely the
spaces $\HFWFLplp$, also known as modulation spaces $\HFMpRd = \HFMppRd$.
They form a chain of strict inclusions, for $1 < p < q < \infty$:
$$
  \HFSOsp = \HFWFLili = \HFMisp \HFhkr \HFMpRd \HFhkr \HFMqRd \HFhkr \HFMinf = \HFWFLinlin = \HFSOPsp. $$
The Segal algebra $\HFSOGN$ has been introduced already in \cite{fe81-2}
and can be considered as the ``mother of all modulation spaces'', due
to its many characterizations, see \cite{ja18} and good functorial
properties (\cite{lo83-1}). As a Segal algebra it is a Banach ideal
in $\HFLiRdN$, satisfying
$$ \HFSOnorm{g \ast f} \leq \HFlinorm{g} \HFSOnorm{f}, \quad g \in \Lisp, f \in \HFSOsp, $$
and due its Fourier invariance the corresponding estimate holds
true for pointwise multiplication with elements from the Fourier
algebra, i.e.
$$ \HFSOnorm{ h \cdot f}  \leq \HFnormta {h} {\HFFLisp} \HFSOnorm{f},
\quad  h \in \HFFLisp, f \in \HFSOsp. $$
In fact, it has been introduced as the smallest among all
Banach spaces (say inside of $\HFLtRd$) which has this
double module property (in the sense of \cite{brfe83}),
i.e., satisfying these two estimates. Correspondingly
$\HFSOPsp$ is the largest such Banach spaces, even in the
context of tempered distributions.

Together with the Hilbert space  $\HFLtRdN$ these
three Wiener amalgam spaces (but also modulation spaces) satisfy
$$  \HFSORdN \HFhkr \HFLtRdN \HFhkr \HFSOPRdN, $$ and form the
so-called {\it Banach Gelfand Triple} $\HFSOGTrRd$. The basic
reference here is \cite{cofelu08}.

The setting of this Banach Gelfand Triple has been found
very useful in the context of time-frequency and in particular
Gabor Analysis. One can say that for a good Gabor system,
say obtained by using a Gaussian window and lattice constants $a,b$ with
$ab < 1$ these Banach spaces correspond to the collections
of all those tempered distributions with (canonical) Gabor
coefficients in $\HFellGTr(\HFZtdst)$ respectively. But this
setting is not only helpful for the analysis of such kind
of questions, it can also be used for novel interpretations
of questions of classical Fourier Analysis, see \cite{fe19,fe23}.

Among others this setting allows to prove the so-called
kernel theorem, i.e. the description of bounded linear
operators from $\HFSORdN$ to $\HFSOPRdN$ (thus e.g. from
$\HFLpRdN$ to $\HFLqRdN$ for $1 \leq p < \infty$ and
$1 \leq q \leq \infty$) by the ``integral kernel''
$K(x,y) \in \HFSOPRtd$. The inner kernel theorem,
i.e., the careful analysis of those operators with
kernels $K \in \HFSORtd$ is carried out in \cite{feja22}.
Unfortunately this paper contains a slight gap. Not
every $\HFwstd$to-norm continuous mapping from
$\HFSOPRd$ to $\HFSORdN$ has a kernel $K \in \HFSORtd$. Instead,
 these operators can be identified with called the $\HFwstd$nuclear operators
from $\HFSOPRd$ to $\HFSORd$, but this correction will be
discussed in a separate note. It will be no surprise that
the classical kernel theorem identifying  $\HFLtRtd$
with the Hilbert-Schmidt operators appears as a restriction
of the general kernel theorem.

Let us conclude this short session with the hint
that the usual approach to the {\it regularization}
of tempered distributions, by applying a smoothing
operator (by convolution with a test function) and then
a pointwise multiplication (by another test function) providing
localization, can be applied in a similar fashion in the
setting of $\HFSOGTrRd$.

Recall that we have
$$
\HFScRd \cdot (\HFScRd \ast \HFScPRd) \subseteq \HFScRd  \HFqandq \HFScRd \ast (\HFScRd \cdot \HFScPRd) \subseteq \HFScRd,
$$ 
and correspondingly
$$ \HFSORd \cdot (\HFSORd \ast \HFSOPRd) \subseteq \HFSORd  \HFqandq
\HFSORd \ast (\HFSORd \cdot \HFSOPRd) \subseteq \HFSORd.
$$ 
Whereas the first set of inclusions requires a careful
treatment of derivative and a lot of analysis, the corresponding
statements involving $\HFSORd$ can be derived in essentially
in two lines,  using standard convolution and multiplication results
for Wiener amalgam spaces. Let us show how this can be done
for the left hand inclusion:
\begin{equation}
\HFSOsp \ast \HFSOPsp = \HFWFLili \ast \HFWFLinlin \subseteq \HFWFLilin
\end{equation}
because $ \HFFLisp \ast \HFFLinsp \subset \HFFLisp$ is just the
Fourier version of the obvious pointwise product rule
$ \Lisp \cdot \HFLinsp \subseteq \Lisp$.
Using the fact that $\HFFLisp$ is a pointwise algebra we can
now apply a  pointwise multiplication result by $\HFSOsp$, obtaining thus
\begin{equation}
\HFSOsp \cdot \HFWFLilin=\HFWFLili \cdot \HFWFLilin \subseteq \HFWFLili = \HFSOsp.
\end{equation}

\section{Conclusion}

Overall one can say that Wiener amalgams play an important role in
many areas of mathematical analysis. Among the function spaces described
in \cite{fe15} they are probably the overall most useful family of Banach
spaces of distributions. Like the principles of complex interpolation theory
the basic results (duality, convolution, Fourier transforms) are easy to
use, even if one has not studied the underlying technical (although somehow
elementary)  principles.
\bibliographystyle{abbrv}  

\end{document}